\documentclass{amsart}
\usepackage{amssymb}
\title{Graph Derangements}
\author{Pete L. Clark}

\thanks{Thanks to Josh Laison and the rest of the Willamette University mathematics department for their interest in this work and for their generous funding and hosting of my visit in March 2012.}

\begin{document}
\newtheorem{lemma}{Lemma}
\newtheorem{prop}[lemma]{Proposition}
\newtheorem{cor}[lemma]{Corollary}
\newtheorem{thm}[lemma]{Theorem}
\newtheorem{ques}{Question}
\newtheorem{quest}[ques]{Question}
\newtheorem{conj}[lemma]{Conjecture}
\newtheorem{fact}[lemma]{Fact}
\newtheorem*{mainthm}{Main Theorem}
\newtheorem{obs}[lemma]{Observation}
\newtheorem{hint}{Hint}
\newtheorem{prob}{Problem}
\maketitle
\newcommand{\pp}{\mathfrak{p}}
\renewcommand{\gg}{\mathfrak{g}}
\newcommand{\DD}{\mathcal{D}}
\newcommand{\F}{\ensuremath{\mathbb F}}
\newcommand{\Fp}{\ensuremath{\F_p}}
\newcommand{\Fl}{\ensuremath{\F_l}}
\newcommand{\Fpbar}{\overline{\Fp}}
\newcommand{\Fq}{\ensuremath{\F_q}}
\newcommand{\PP}{\mathbb{P}}
\newcommand{\PPone}{\mathfrak{p}_1}
\newcommand{\PPtwo}{\mathfrak{p}_2}
\newcommand{\PPonebar}{\overline{\PPone}}
\newcommand{\N}{\ensuremath{\mathbb N}}
\newcommand{\Q}{\ensuremath{\mathbb Q}}
\newcommand{\Qbar}{\overline{\Q}}
\newcommand{\R}{\ensuremath{\mathbb R}}
\newcommand{\Z}{\ensuremath{\mathbb Z}}
\newcommand{\SSS}{\ensuremath{\mathcal{S}}}
\newcommand{\Rn}{\ensuremath{\mathbb R^n}}
\newcommand{\Ri}{\ensuremath{\R^\infty}}
\newcommand{\C}{\ensuremath{\mathbb C}}
\newcommand{\Cn}{\ensuremath{\mathbb C^n}}
\newcommand{\Ci}{\ensuremath{\C^\infty}}
\newcommand{\U}{\ensuremath{\mathcal U}}
\newcommand{\gn}{\ensuremath{\gamma^n}}
\newcommand{\ra}{\ensuremath{\rightarrow}}
\newcommand{\fhat}{\ensuremath{\hat{f}}}
\newcommand{\ghat}{\ensuremath{\hat{g}}}
\newcommand{\hhat}{\ensuremath{\hat{h}}}
\newcommand{\covui}{\ensuremath{\{U_i\}}}
\newcommand{\covvi}{\ensuremath{\{V_i\}}}
\newcommand{\covwi}{\ensuremath{\{W_i\}}}
\newcommand{\Gt}{\ensuremath{\tilde{G}}}
\newcommand{\gt}{\ensuremath{\tilde{\gamma}}}
\newcommand{\Gtn}{\ensuremath{\tilde{G_n}}}
\newcommand{\gtn}{\ensuremath{\tilde{\gamma_n}}}
\newcommand{\gnt}{\ensuremath{\gtn}}
\newcommand{\Gnt}{\ensuremath{\Gtn}}
\newcommand{\Cpi}{\ensuremath{\C P^\infty}}
\newcommand{\Cpn}{\ensuremath{\C P^n}}
\newcommand{\lla}{\ensuremath{\longleftarrow}}
\newcommand{\lra}{\ensuremath{\longrightarrow}}
\newcommand{\Rno}{\ensuremath{\Rn_0}}
\newcommand{\dlra}{\ensuremath{\stackrel{\delta}{\lra}}}
\newcommand{\pii}{\ensuremath{\pi^{-1}}}
\newcommand{\la}{\ensuremath{\leftarrow}}
\newcommand{\gonem}{\ensuremath{\gamma_1^m}}
\newcommand{\gtwon}{\ensuremath{\gamma_2^n}}
\newcommand{\omegabar}{\ensuremath{\overline{\omega}}}
\newcommand{\dlim}{\underset{\lra}{\lim}}
\newcommand{\ilim}{\operatorname{\underset{\lla}{\lim}}}
\newcommand{\Hom}{\operatorname{Hom}}
\newcommand{\Ext}{\operatorname{Ext}}
\newcommand{\Part}{\operatorname{Part}}
\newcommand{\Ker}{\operatorname{Ker}}
\newcommand{\im}{\operatorname{im}}
\newcommand{\ord}{\operatorname{ord}}
\newcommand{\unr}{\operatorname{unr}}
\newcommand{\B}{\ensuremath{\mathcal B}}
\newcommand{\Ocr}{\ensuremath{\Omega_*}}
\newcommand{\Rcr}{\ensuremath{\Ocr \otimes \Q}}
\newcommand{\Cptwok}{\ensuremath{\C P^{2k}}}
\newcommand{\CC}{\ensuremath{\mathcal C}}
\newcommand{\gtkp}{\ensuremath{\tilde{\gamma^k_p}}}
\newcommand{\gtkn}{\ensuremath{\tilde{\gamma^k_m}}}
\newcommand{\QQ}{\ensuremath{\mathcal Q}}
\newcommand{\I}{\ensuremath{\mathcal I}}
\newcommand{\sbar}{\ensuremath{\overline{s}}}
\newcommand{\Kn}{\ensuremath{\overline{K_n}^\times}}
\newcommand{\tame}{\operatorname{tame}}
\newcommand{\Qpt}{\ensuremath{\Q_p^{\tame}}}
\newcommand{\Qpu}{\ensuremath{\Q_p^{\unr}}}
\newcommand{\scrT}{\ensuremath{\mathfrak{T}}}
\newcommand{\That}{\ensuremath{\hat{\mathfrak{T}}}}
\newcommand{\Gal}{\operatorname{Gal}}
\newcommand{\Aut}{\operatorname{Aut}}
\newcommand{\tors}{\operatorname{tors}}
\newcommand{\Zhat}{\hat{\Z}}
\newcommand{\linf}{\ensuremath{l_\infty}}
\newcommand{\Lie}{\operatorname{Lie}}
\newcommand{\GL}{\operatorname{GL}}
\newcommand{\End}{\operatorname{End}}
\newcommand{\aone}{\ensuremath{(a_1,\ldots,a_k)}}
\newcommand{\raone}{\ensuremath{r(a_1,\ldots,a_k,N)}}
\newcommand{\rtwoplus}{\ensuremath{\R^{2  +}}}
\newcommand{\rkplus}{\ensuremath{\R^{k +}}}
\newcommand{\length}{\operatorname{length}}
\newcommand{\Vol}{\operatorname{Vol}}
\newcommand{\cross}{\operatorname{cross}}
\newcommand{\GoN}{\Gamma_0(N)}
\newcommand{\GeN}{\Gamma_1(N)}
\newcommand{\GAG}{\Gamma \alpha \Gamma}
\newcommand{\GBG}{\Gamma \beta \Gamma}
\newcommand{\HGD}{H(\Gamma,\Delta)}
\newcommand{\Ga}{\mathbb{G}_a}
\newcommand{\Div}{\operatorname{Div}}
\newcommand{\Divo}{\Div_0}
\newcommand{\Hstar}{\cal{H}^*}
\newcommand{\txon}{\tilde{X}_0(N)}
\newcommand{\sep}{\operatorname{sep}}
\newcommand{\notp}{\not{p}}
\newcommand{\Aonek}{\mathbb{A}^1/k}
\newcommand{\Wa}{W_a/\mathbb{F}_p}
\newcommand{\Spec}{\operatorname{Spec}}
\newcommand{\MaxSpec}{\operatorname{MaxSpec}}

\newcommand{\abcd}{\left[ \begin{array}{cc}
a & b \\
c & d
\end{array} \right]}

\newcommand{\abod}{\left[ \begin{array}{cc}
a & b \\
0 & d
\end{array} \right]}

\newcommand{\unipmatrix}{\left[ \begin{array}{cc}
1 & b \\
0 & 1
\end{array} \right]}

\newcommand{\matrixeoop}{\left[ \begin{array}{cc}
1 & 0 \\
0 & p
\end{array} \right]}

\newcommand{\w}{\omega}
\newcommand{\Qpi}{\ensuremath{\Q(\pi)}}
\newcommand{\Qpin}{\Q(\pi^n)}
\newcommand{\pibar}{\overline{\pi}}
\newcommand{\pbar}{\overline{p}}
\newcommand{\lcm}{\operatorname{lcm}}
\newcommand{\trace}{\operatorname{trace}}
\newcommand{\OKv}{\mathcal{O}_{K_v}}
\newcommand{\Abarv}{\tilde{A}_v}
\newcommand{\kbar}{\overline{k}}
\newcommand{\Kbar}{\overline{K}}
\newcommand{\pl}{\rho_l}
\newcommand{\plt}{\tilde{\pl}}
\newcommand{\plo}{\pl^0}
\newcommand{\Du}{\underline{D}}
\newcommand{\A}{\mathbb{A}}
\newcommand{\D}{\underline{D}}
\newcommand{\op}{\operatorname{op}}
\newcommand{\Glt}{\tilde{G_l}}
\newcommand{\gl}{\mathfrak{g}_l}
\newcommand{\gltwo}{\mathfrak{gl}_2}
\newcommand{\sltwo}{\mathfrak{sl}_2}
\newcommand{\h}{\mathfrak{h}}
\newcommand{\tA}{\tilde{A}}
\newcommand{\sss}{\operatorname{ss}}
\newcommand{\X}{\Chi}
\newcommand{\ecyc}{\epsilon_{\operatorname{cyc}}}
\newcommand{\hatAl}{\hat{A}[l]}
\newcommand{\sA}{\mathcal{A}}
\newcommand{\sAt}{\overline{\sA}}
\newcommand{\OO}{\mathcal{O}}
\newcommand{\OOB}{\OO_B}
\newcommand{\Flbar}{\overline{\F_l}}
\newcommand{\Vbt}{\widetilde{V_B}}
\newcommand{\XX}{\mathcal{X}}
\newcommand{\GbN}{\Gamma_\bullet(N)}
\newcommand{\Gm}{\mathbb{G}_m}
\newcommand{\Pic}{\operatorname{Pic}}
\newcommand{\FPic}{\textbf{Pic}}
\newcommand{\solv}{\operatorname{solv}}
\newcommand{\Hplus}{\mathcal{H}^+}
\newcommand{\Hminus}{\mathcal{H}^-}
\newcommand{\HH}{\mathcal{H}}
\newcommand{\Alb}{\operatorname{Alb}}
\newcommand{\FAlb}{\mathbf{Alb}}
\newcommand{\gk}{\mathfrak{g}_k}
\newcommand{\car}{\operatorname{char}}
\newcommand{\Br}{\operatorname{Br}}
\newcommand{\gK}{\mathfrak{g}_K}
\newcommand{\coker}{\operatorname{coker}}
\newcommand{\red}{\operatorname{red}}
\newcommand{\CAY}{\operatorname{Cay}}
\newcommand{\ns}{\operatorname{ns}}
\newcommand{\xx}{\mathbf{x}}
\newcommand{\yy}{\mathbf{y}}
\newcommand{\E}{\mathbb{E}}
\newcommand{\rad}{\operatorname{rad}}
\newcommand{\Top}{\operatorname{Top}}
\newcommand{\Sym}{\operatorname{Sym}}
\newcommand{\Cl}{\operatorname{Cl}}
\newcommand{\Der}{\operatorname{Der}}
\newcommand{\Perm}{\operatorname{Perm}}


\begin{abstract}
We introduce the notion of a \textbf{graph derangement}, which naturally interpolates between perfect matchings and Hamiltonian cycles.  We give a necessary and sufficient condition for the existence of graph derangements on a locally finite graph.  This result was first proved by W.T. Tutte in 1953 by applying some deeper results on digraphs.  We give a new, simple proof which amounts to a reduction to the (Menger-Egerv\'ary-K\"onig-)Hall(-Hall) Theorem on transversals of set systems.  Finally, we consider the problem of classifying all cycle types of graph derangements on $m \times n$ checkerboard graphs.  Our presentation \emph{does not} assume any prior knowledge in graph theory or combinatorics: all definitions and proofs of needed theorems are given.
\end{abstract}

\section{Introduction}

\subsection{The Cockroach Problem}
\textbf{} \\ \\
Consider a square kitchen floor tiled by $25 = 5 \times 5$ square tiles in the usual manner.  Late one night the house's owner comes down to discover that the floor is crawling with cockroaches: in fact each square tile contains a single cockroach.  Displeased, she goes to the kitchen cabinet and pulls out an enormous can of roach spray.  The roaches sense what is coming and start skittering.  Each roach has enough time to skitter from its own tile to any adjacent tile.  But it will not be so good for two (or more) roaches to skitter to the same tile: that will make an obvious target.  Is it indeed possible for the roaches to perform a collective skitter in this way?
\\ \\
This is a nice problem to give to undergraduates: it is concrete, 
fun, and far away from what they think they should be doing 
in a math class.  

\subsection{Solution}
\textbf{} \\ \\
Suppose the tiles are painted black and white with a checkerboard pattern, 
and that the center square is black, so that there are $13$ black squares and $12$ white squares.  Therefore there are $13$ roaches 
who start out on black squares and are seeking a home on only $12$ white squares.  It is not possible -- no more so than for pigeons! -- for all 
$13$ cockroaches to end up on different white squares.  

\subsection{A problem for mathematicians}
\textbf{} \\ \\
For a grown mathematician (or even an old hand at mathematical brainteasers) this is not a very challenging problem, since 
the above parity considerations will quickly leap to mind.  Nevertheless 
there is something about it that encourages further contemplation.  There 
were several other mathematicians at the Normal Bar and they were paying 
attention too.  ``What about the $6 \times 6$ case?'' one of them asked.  
``It reminds me of the Brouwer fixed point theorem,'' muttered another.\footnote{Unfortunately, a connection with Brouwer's theorem is \emph{not} achieved here, but see \cite{Knill}.}   
\\ \\
One natural followup is to ask what happens for cockroaches 
on an $m \times n$ rectangular grid.  The preceding argument works when 
$m$ and $n$ are both odd.  On the other hand, if e.g. $m = n = 2$ it clearly is 
possible for the cockroaches to skitter, and already there are several different ways.  For instance, we could divide the rectangle into two dominos and have 
the roaches on each domino simply exchange places.
Or we could simply 
have them proceed in a (counter)clockwise cycle.  \\ \indent
Dominos are a good idea in general: if one of $m$ and $n$ is even, then an 
$m \times n$ rectangular grid may be tiled with dominos, and this gives a way for the roaches to skitter.  Just as above though this feels not completely satisfactory and one naturally looks for other skittering patterns: e.g. when $m = n = 4$ 
one can have the roaches in the inner $2 \times 2$ square skittering clockwise 
as before and then roaches in the outer ring of the square skittering around in 
a cycle: isn't that more fun?  There are many other skittering patterns as well.
\\ \\
I found considerations like the above to be rather interesting (and I will come 
back to them later), but for me the real problem was a bit more meta: what is the mathematical structure underlying the Cockroach Problem, and what is the general question being asked about this structure?
\\ \indent
Here we translate the Cockroach Problem into graph-theoretic terms.  In so doing we get a graph theoretic 
problem which in a precise sense interpolates between two famous and classic problems: existence of \textbf{perfect matchings} and existence of \textbf{Hamiltonian cycles}.  On the other hand, the more general problem does 
not seem to be well known.  But it's interesting, and we present it here: it is the existence and classification of \textbf{graph derangements}.

\section{Graph derangements and graph permutations}

\subsection{Basic definitions}
\textbf{} \\ \\ \noindent
Let $G = (V,E)$ be a simple, undirected graph: that is, we 
are given a finite set $V$ of \textbf{vertices} and a set $E$ of \textbf{edges}, which are unordered pairs of distinct elements of $V$.  For $v_1,v_2 \in V$, we 
say that $v_1$ and $v_2$ are \textbf{adjacent} if $\{v_1,v_2\} \in E$ and 
write $v_1 \sim v_2$.  In other words, for a set $V$, to give 
a graph with vertex set $V$ is equivalent to giving an anti-reflexive, symmetric 
binary relation on $V$, the \textbf{adjacency relation}.  A variant formalism is also useful: we may think of a 
graph as a pair of sets $(V,E)$ and an \textbf{incidence relation} 
on $V \times E$.  Namely, for 
$x \in V$ and $e \in E$, $x$ is \textbf{incident} to $e$ if $x \in e$, or, 
less formally, if $x$ is one of the two vertices comprising the endpoints of $e$.  If one knows the incidence relation as a subset of $V \times E$ 
then one knows in particular for each $e \in E$ the pair of vertices $\{v_1,v_2\}$ which are incident to $E$ and thus one knows the graph $G$.
\\ \\
For $G = (V,E)$ and $v \in V$, the \textbf{degree} of $v$ is the number of edges which are incident to $v$.  A degree zero vertex is \textbf{isolated}; a degree one vertex is \textbf{pendant}.
\\ \\
A graph is \textbf{finite} if its vertex set is finite (and hence its edge set is finite as well).  A graph is \textbf{locally finite} if every vertex has finite degree.  
\\ \\
If $G = (V,E)$ and $G' = (V,E')$ are graphs with $E' \subset E$, 
we say $G'$ is an \textbf{edge subgraph} of $G$: it has the same underlying vertex set as $G$ and is obtained from $G$ 
by removing some edges.  If $G = (V,E)$ and $G' = (V',E')$ are finite graphs with $V' \subset V$ and $E'$ the set of all elements of $E$ linking two 
vertices in $V'$, we say $G'$ is an \textbf{induced subgraph} of $G$.  
\\ \\
For a graph $G = (V,E)$, a subset $X \subset V$ is \textbf{independent} 
if for no $x_1,x_2 \in X$ do we have $x_1 \sim x_2$.
\\ \\
Example 2.1: For $m,n \in \Z^+$, we formally define the \textbf{checkerboard graph} $R_{m,n}$.  Its vertex set is $\{1,\ldots,m\} \times \{1,\ldots,n\}$ and we decree that $(x_1,x_2) \sim (y_1,y_2)$ if $|x_1-y_1| + |x_2 - y_2| = 1$.  
\\ \\
Example 2.2: More generally, for $n \in \Z^+$ and $m_1,\ldots,m_n \in \Z^+$ 
we may define an $n$-dimensional analogue $R_{m_1,\ldots,m_n}$ of $R_{m,n}$.  Its vertex set 
is $\{1,\ldots,m_1\} \times \ldots \times \{1,\ldots,m_n\}$ and we 
decree that $(x_1,\ldots,x_m) \sim (y_1,\ldots,y_m)$ if $\sum_{i=1}^m 
|x_i - y_i| = 1$.
\\ \\
Example 2.3: For $m \geq 2, n \geq 1$ we define the \textbf{M\"obius 
checkerboard gaph} $M_{m,n}$.  This is a graph with the same vertex set as $R_{m,n}$ and having edge set consisting of all the edges of $R_{m,n}$ together 
with $(x,n) \sim (x,1)$ for all $1 \leq x \leq m$.  We put $C_m = M_{m,1}$, 
the \textbf{cycle graph}.  The checkerboard graph $R_{m,n}$ is an edge 
subgraph of the M\"obius checkerboard graph $M_{m,n}$.  
\\ \\
Example 2.4: For $m,n \geq 2$ we define the \textbf{torus checkerboard graph} 
$T_{m,n}$.  Again it has the same vertex set as $R_{m,n}$ and contains all 
of the edges of $R_{m,n}$ together with the following ones: for all $1 \leq x \leq m$, $(x,n) \sim (x,1)$ and for all $1 \leq y \leq n$, $(m,y) \sim (1,y)$.  
The M\"obius checkerboard graph $M_{m,n}$ is an edge subgraph of the 
torus checkerboard graph $T_{m,n}$.    
\\ \\
For a graph $G = (V,E)$ and $x \in V$, we define the \textbf{neighborhood
of x} as $ N_x = \{y \in V \ | \ x \sim y \}$.  More generally, for any subset $X \subset V$ we define the \textbf{neighborhood of X} as $N(X) = \{y \in V \ | \ \exists x \in X \text{ such that } x \sim y \} = \bigcup_{x \in X} N_x$.
\\ \\
Remark 2.5: Although $x \notin N_x$, $X$ and $N(X)$ \emph{need not} be disjoint.  
In fact, $X \cap N(X) = \varnothing$ iff $X$ is an independent set.   
\\ \\
Now the ``cockroach skitterings'' that 
we were asking about on $R_{m,n}$ can be enunciated in terms of any graph  as follows.
\\ \\
Let $G = (V,E)$ be a graph.  A \textbf{graph derangement} of $G$ 
is an injection $f: V \ra V$ with $f(v) \sim v$ for all $v \in V$.  Let $\Der G$ be the set of all graph derangements of $G$.  
\\ \\
It is natural to also consider a slightly more general definition.
\\ \\
For a graph $G = (V,E)$, a \textbf{graph permutation} of $G$ 
is an injection $f: V \ra V$ such that for all $v \in V$, either $f(v) \sim v$ 
or $f(v) = v$.  Let $\Perm G$ be the set of all graph permutations of $G$.  Thus a graph derangement is precisely a graph permutation which is fixed point 
free: for all $v \in V$, $f(v) \neq v$. 
\begin{lemma}
\label{2.0}
Let $G = (V,E)$ be a graph.  \\
a) If $G$ has an isolated vertex, $\Der G = \varnothing$.  \\
b) If $G$ has two pendant vertices adjacent to a common vertex, $\Der G = \varnothing$.  
\end{lemma}
\begin{proof}
Left to the reader as an exercise to get comfortable with the definitions.
\end{proof}
\noindent
Remark 2.6: If $G'$ is an edge subgraph of $G$, any graph derangement (resp. graph permutation) $\sigma$ of $G'$ is also a graph derangement (resp. graph permutation) of $G$.  
%

\subsection{Cycles and surjectivity}
\textbf{} \\ \\ \noindent
Given a graph $G$, we wish not only to decide whether $\Der G$ is nonempty but also to study its structure.  The collection of all derangements on a given graph is 
likely to be a very complicated object: consider for instance $\Der K_n$, 
which has size asymptotic to $\frac{n!}{e}$.  Just as in the case of ordinary 
permutations and derangements, it seems interesting to study the possible 
\textbf{cycle types} of graph derangements and graph permutations on a given graph $G$.  Let us give careful definitions of these.
\\ \indent 
First, let $V$ be a set and $f: V \ra V$ an function.  For $m \in \Z^+$, 
let $f^m = f \circ \ldots \circ f$ be the \emph{mth iterate} of $f$.  We introduce a 
relation $\approx$ on $V$ as follows: for $x,y \in V$, $x \approx y$ iff 
there are $m,n \in \Z^+$ such that $f^m x = f^n y$.  This is an equivalence relation on $V$: the reflexivity and the symmetry are 
immediate, and as for the transitivity: if $x,y,z \in V$ are such that 
$x \approx y$ and $y \approx z$ then there are $a,b,c,d \in \Z^+$ with 
$f^a x = f^b y$ and $f^c y = f^d z$, and then 
\[f^{a+c} x = f^c \circ f^a x = f^c \circ f^b y =  f^{b+c} y = f^b \circ f^c y =  f^b \circ 
f^d z = f^{b+d} z. \]
Now suppose $f$ is injective: we now call the $\approx$-equivalence classes \textbf{cycles}.  Let $x \in V$, and denote the cycle containing $x$ by $C_x$.  
Then: \\
$\bullet$ $C_x$ is \textbf{finite} iff $x$ lies in the image of $f$ and there is 
$m \in \Z^+$ such that $f^m x = x$.  \\
$\bullet$ $C_x$ is \textbf{singly infinite} iff it is infinite and there are 
$y \in V$, $m \in \Z^+$ such that $f^m y = x$ and $y$ is not in the image 
of $f$.  \\
$\bullet$ $C_x$ is \textbf{doubly infinite} iff it is infinite and every $y \in C_x$ lies in the image of $f$.\\  
These cases are mutually exclusive and exhaustive, so 
$f$ is surjective iff there are no singly infinite cycles.  Suppose $G = (V,E)$ is a graph and $f \in \Perm G$.  We can define the 
\textbf{cycle type} of $f$ as a map from the set of 
possible cycle types into the 
class of cardinal numbers.  When $V$ is finite, this amounts to a partition of 
$\# V$ in the usual sense: e.g. the cycle type of roaches skittering counterclockwise on a $5 \times 5$ grid is $(1,8,16)$.  A graph permutation is a 
graph derangement if it has no $1$-cycles.  We will say a graph derangement is \textbf{matchless} if it has no $2$-cycles.

\begin{prop}
\label{ET2}
Suppose that a graph $G$ admits a graph derangement.  Then $G$ admits a surjective graph derangement.
\end{prop}
\begin{proof} 
It is sufficient to show that any graph derangement can be modified to yield a graph derangement with no singly infinite cycles, and for that it suffices to 
consider one singly infinite cycle, which may be 
viewed as the derangement $n \mapsto n+1$ on the graph $\Z^+$ with $n \sim n+1$ for all $n \in \Z^+$.  This derangement can be decomposed 
into an infinite union of $2$-cycles: $1 \leftrightarrow 2, \ 3 \leftrightarrow 4,  \ldots, 2n-1 \leftrightarrow 2n,  \ldots$.
\end{proof}

\subsection{Disconnected Graphs}

\begin{prop}
\label{2.3}
Let $G$ be a graph with components $\{G_i\}_{i \in I}$, and let $f \in \Perm G$.  \\
a) For all $i \in I$, $f(G_i) = G_i$.  \\
b) Conversely, given graph permutations $f_i$ on each $G_i$, $f = \bigcup_{i \in I} f_i: G \ra G$ is a graph permutation.  Moreover $f \in \Der G \iff f_i \in \Der G_i$ for all $i \in I$. 
\end{prop}
\noindent
The proof is immediate.  Thus we may as well restrict attention to connected graphs.

\subsection{Bipartite Graphs}
\textbf{} \\ \\ \noindent
A \textbf{bipartition} of a graph $G = (V,E)$ is a partition $V = V_1 \coprod V_2$ of the vertex set such that each $V_i$ is an independent set.  A graph is \textbf{bipartite} if it admits at least one bipartition.
\\ \\  
For $k \in \Z^+$, a \textbf{k-coloring} of a graph $G = (V,E)$ is a map 
$C: V \ra \{1,\ldots,k\}$ such that for all $x,y \in V$, $x \sim y \implies C(x) \neq C(y)$.  There is a bijective correspondence 
between $2$-colorings of $G$ and bipartitions of $G$: given a $2$-coloring 
$C$ we define $V_i = \{x \in V \ | \ C(x) = i \}$, and given a bipartition 
we define $C(x) = i$ if $x \in V_i$.  Thus a graph is bipartite 
iff it admits a $2$-coloring.
\\ \\
Remark 2.7: For a graph $G = (V,E)$, a map $C: V \ra \{0, 1 \}$ is a 
$2$-coloring of $G$ iff its restriction to each connected component 
$G_i$ is a $2$-coloring of $G_i$.  It follows that a graph is bipartite 
iff all of its connected components are bipartite.
\\ \\
Remark 2.8: Any subgraph $G'$ of a bipartite graph $G$ is bipartite.  Indeed, 
any $2$-coloring of $G$ restricts to a $2$-coloring of $G'$.
\\ \\
Example 2.9: The cycle graph $C_m$ is bipartite iff $m$ is even.  

\begin{cor}
\label{2.5}
Let $G$ be a bipartite graph, and let 
$\sigma \in Der G$.  Then every finite cycle of $\sigma$ has even length.
\end{cor}
\begin{proof} Since a subgraph of a bipartite graph is bipartite, a 
bipartite graph cannot admit a cycle of odd finite degree.\footnote{Conversely, a graph with no cycles of odd degree is bipartite, a famous result 
of D. K\H onig.}
\end{proof}

\subsection{Some Examples}
\textbf{} \\ \\ \noindent
Example 2.10: If $G = K_n$ is the complete graph on the vertex set $[n] = \{1,\ldots,n\}$, then a graph permutation of $G$ is nothing else than a permutation of $[n]$.  A derangement of $K_n$ is a derangement in the usual 
sense, i.e., a fixed-point free permutation.  Derangements exist 
iff $n > 1$ and the number of them is asymptotic to $\frac{n!}{e}$ as $n \ra \infty$.
\\ \\
Example 2.11: \\
a) The $n$-dimensional checkerboard graph $R_{m_1,\ldots,m_n}$ 
admits a graph derangement iff $m_1,\ldots,m_n$ are not all odd.  \\
b) The M\"obius checkerboard graphs $M_{m,n}$ all admit graph derangements.  Hence, by Remark 2.6, so do the torus checkerboard graphs $T_{m,n}$.  \\
c) For odd $n$, the square checkerboard graph $R_{n,n}$ admits graph permutation with a single fixed point.  For instance, by dividing the square into concentric rings we get a graph permutation with cycle type $\{1,8,16,\ldots,8 \left(\frac{n-1}{2} \right)\}$.  
%
%

\section{Existence Theorems}

\subsection{Halls' Theorems}
\textbf{} \\ \\ \noindent
The main tool in all of our Existence Theorems is a truly basic result of 
combinatorial theory.  There are several (in fact, notoriously many) equivalent 
versions, but for our purposes it will be helpful to single out two different 
formulations.

\begin{thm}(Halls' Theorem: Transversal Form)
\label{HMT}
\label{2.7}
Let $V$ be a set and $\mathcal{S} = \{S_i\}_{i \in I}$ be an indexed family of finite subsets of $V$.  The following are equivalent: \\
(i) (Hall Condition) For every finite subfamily $J \subset I$, 
\[ \# J \leq \# \bigcup_{i \in J} S_i. \]
(ii) $(V,I)$ admits a \textbf{transversal}: a subset $X \subset V$ 
and a bijection $f: X \ra I$ such that for all $x \in X$, $x \in S_{f(x)}$.
\end{thm}
\noindent
We will deduce Theorem \ref{HMT} from the following reformulation.

\begin{thm}(Halls' Theorem: Marriage Form)
\label{HMTII}
Let $G = (V_1,V_2,E)$ be a bipartitioned graph in which every vertex in $V_1$ has finite degree.  TFAE: \\
(i) (Cockroach Condition) For every finite subset of $V_1$, $\# V_1 \leq \# N(V_1)$. \\
(ii) There is a \textbf{semiperfect matching}, that is an injection $\iota: V_1 \ra V_2$ such that for all $x \in V_1$, $x \sim \iota(x)$.
\end{thm}
\begin{proof} \cite{Halmos-Vaughan} Step 1: Suppose $V_1$ is finite.  We go by induction on $\# V$.  The case $\# V_1 = 1$ is trivial.  Now suppose that $\# V_1 = n > 1$ and that the result holds for all bipartitioned graphs with first vertex set of cardinality smaller than $n$.  It will be notationally convenient 
to suppose that $V_1 = \{1,\ldots,n\}$, and we do so.  
\\
Case 1: Suppose that for all $1 \leq k < n$, every $k$-element subset of $V_1$ 
has at least $k+1$ neighbors.  Then we may match $n$ to any element of $V_2$ and 
semiperfectly match $\{1,\ldots,n-1\}$ into the remaining elements of $V_2$ by induction.  \\
Case 2: Otherwise, for some $k$, $1 \leq k < n$, there is a $k$-element subset 
$X \subset V_1$ such that $\# N(X) = k$.  The subset $X$ may be semiperfectly 
matched into $V_2$ by induction, say via $\iota_1: X \ra V_2$, so it suffices to show that the Hall Condition still holds on the induced bipartitioned subgraph on $(V_1 \setminus X,V_2 \setminus \iota_1(X))$.  Indeed, if not, then for 
some $h$, $1 \leq h \leq n-k$, there would be an $h$-element subset $Y \subset V_1 \setminus X$ having fewer than $h$ neighbors in $V_2 \setminus \iota_1(X))$, 
but then \[\# N(X \cup Y) = \# (N(X) \cup N(Y)) \leq \# N(X) + \# N(Y) < k + h = 
\# (X \cup Y). \]
Step 2: Suppose $V_1$ is infinite.  For $x \in V_1$, endow $N_x$ with the discrete topology; endow 
$N = \prod_{x \in V_1} N_x$ with the product topology.  Each $N_x$ is finite hence compact, so $N$ is compact by Tychonoff's Theorem.  For any finite subset $X \subset V_1$, let \[H_X = \{ n = \{n_i\} \in N \ | \ 
n_{x} \neq n_{y} \ \forall x \neq y \in X\}. \]
Then $H_X$ is closed in $N$ and is nonempty by Step 1.  
Since $G$ is compact, there is $n \in \bigcap_X H_X$, and any such $n$ is a semiuperfect 
matching.  
\end{proof}
\noindent
Remark 3.1: Theorems \ref{2.7} and \ref{HMTII} are \emph{equivalent results}: \\ \indent
Assume Theorem \ref{2.7}.  In the setting of Theorem \ref{HMTII} take 
$V = V_2$, $I = V_1$ and $\mathcal{S} = \{N_x\}_{x \in I}$.  The local finiteness of the graph means each element of $\mathcal{S}$ is finite, and the assumed Cockroach Condition is precisely the Hall Condition, so by Theorem \ref{2.7} there is $X \subset V_2$ and a bijection $f: X \ra I$ such that
$x \in X \implies f(x) \in N_x$.  Let $\iota = f^{-1}: V_1 \ra X$.  Then for $y \in V_1$, let $y = \iota(x)$, so $x = f(y) \sim y = \iota(x)$.  \\ \indent
Assume Theorem \ref{HMTII}.  In the setting of Theorem \ref{2.7} 
take $V_1 = I$, $V_2 = V$, and $E$ the set of pairs $(i,x)$ such that $x \in S_i$.  Since each $S_i$ is finite, the graph is locally finite, and the assumed 
Hall Condition is precisely the Cockroach Condition, so by Theorem \ref{HMTII} 
there is a semiperfect matching $\iota: I \ra V$.  Let $X = f(I)$, and let $f: X \ra I$ be the inverse function.  For $x \in X$, if $i = f(x)$, $x = 
\iota(i)$, so $x \in S_i = S_{f(x)}$.  
\\ \\
Remark 3.2: Theorem \ref{2.7} was first proved for finite $I$ by Philip Hall 
\cite{Hall}.  Eventually it was realized that equivalent or stronger versions of 
P. Hall's Theorem had been proven earlier by Menger \cite{Menger}, Egerv\'ary \cite{Egervary} and K\H onig \cite{Konig31}.  The matrimonial interpretation was introduced some years later by Halmos and Vaughan \cite{Halmos-Vaughan}.  Nevertheless, with typical disregard for history the most common name for the finite form 
of either Theorem \ref{HMT} or \ref{HMTII} is \emph{Hall's Marriage Theorem}.
\\ \\
Remark 3.3: The generalization to arbitrary index sets was given by Marshall Hall, Jr. \cite{Hall48}, whence ``Halls' Theorem'' (i.e., the theorem of more than one Hall). 
\\ \\
Remark 3.4: M. Hall, Jr.'s argument used Zorn's Lemma, which is equivalent to the Axiom of Choice (AC).  The proof supplied above 
uses Tychonoff's Theorem, which is also equivalent to (AC) \cite{Kelley50}.  However, all of our spaces are Hausdorff.  By examining the proof of Tychonoff's Theorem using ultrafilters \cite{Cartan37},
one sees that when the spaces are Hausdorff, by the uniqueness of limits one does not need (AC) but only that every filter can be extended to an ultrafilter (UL).  In turn, (UL) 
is equivalent to the fact that every Boolean ring has a prime ideal (BPIT).  (BPIT) is known to be \emph{weaker} than (AC), hence Halls' Theorem cannot imply (AC). 
\begin{quest}
\label{BPITQUEST}
Does Halls' Theorem imply the Boolean Prime Ideal Theorem?
\end{quest} 
\noindent
Remark 3.5: The use of compactness of a product of finite, discrete spaces  is a clue for the \emph{cognoscenti} that it should be possible to find a nontopological proof using the Compactness Theorem from model theory.  The reader who is interested and knowledgeable about such 
things will enjoy doing so.  The Compactness Theorem 
(and also the Completeness Theorem) is known be equivalent to (BPIT).
\\ \\
Example 3.6 \cite[pp. 288-289]{Everett-Whaples}: Let $V_1$ be the set of non-negative integers and let $V_2$ be the 
set of positive integers.  For all \emph{positive} integers $x \in V_1$, 
we decree that $x$ is adjacent to the corresponding positive integer in $V_2$ 
and to no other elements of $V_2$.  However, we decree that $0 \in V_1$ is 
adjacent to \emph{every} element of $Y$.  It is clear that there is no semiperfect matching, because if we match $0$ to any $n \in V_2$, then the 
corresponding element $n \in V_1$ cannot be matched.  But the Cockroach 
Condition holds: for a finite subset $X \subset V_1$, if $0 \neq X$ then 
$\# N(V_1) = \# V_1$, whereas if $0 \in X$ then $N(V_1) = V_2$.

\subsection{The First Existence Theorem}

\begin{thm}
\label{EXISTENCETHM}
Consider the following conditions on a graph $G = (V,E)$: \\
(D) $\Der G \neq \varnothing$.  \\
(H) For all subsets $X \subset V$, $\# X \leq \# N(X)$.  \\
(H$'$) For all \emph{finite} subsets $X \subset V$, $\# X \leq \# N(X)$.  \\ 
a) Then (D) $\implies$ (H) $\implies$ (H$'$).  \\
b)  If $G$ is locally finite, then (H$'$) $\implies$ (D) and thus (D) $\iff$ (H) $\iff$ (H$'$). 
\end{thm}
\begin{proof}
a) (D) $\implies$ (H): If $\sigma \in \Der G$ and $X \subset V$, then $\sigma: V \ra V$ is an injection with $\sigma(X) \subset N(X)$.  Thus $\# X \leq \# N(X)$.
(H) $\implies$ (H$'$) is immediate.  \\
b) (H$'$) $\implies$ (D) if $G$ is locally finite: for each $x \in V$, let 
$S_x = N_x$, and let $I = \{S_x \}_{x \in V}$.  Since $G$ is locally finite, $I$ is 
an indexed family of finite subsets of $V$.  By assumption, for any finite subfamily $J \subset I$, 
\[ \# J \leq \# N(J) = \# \bigcup_{x \in J} 
N_x = \# \bigcup_{i \in J} S_i:\]
this is the Hall Condition.  Thus by Theorem \ref{HMT}, there is $X \subset V$ and a bijection $f: X \ra V$ such that for all $x \in X$, $x \in N_{f(x)}$, i.e., $x \sim f(x)$.  Let $\sigma = f^{-1}: V \ra X$.  Then for all $y \in V$, $\sigma(y) \sim 
f(\sigma(y)) = y$, so $\sigma \in \Der G$.
\end{proof}
\noindent
Remark 3.7: Example 3.3 shows (H) need not imply (D) without the assumption of local finiteness.  The graph with vertex set $\R$ and such that every $x \in \R$ is adjacent to every integer $n > x$ satisfies (H$'$) but not (H).

\begin{lemma}
\label{ET3}
Let $G$ be a locally finite graph which violates the cockroach condition: there is a finite subset $X \subset V(G)$ such that $\# X > \# N(X)$.  Then there 
is an \emph{independent} subset $Y \subset X$ such that $\# Y > \# N(Y)$.
\end{lemma}
\begin{proof}
Let $Y \subset X$ be the subset of all vertices which are not adjacent to 
any element of $X$, so $Y$ is an independent set.  Put $m_1 = \# Y$, $m_2 = \# (X \setminus Y)$, and $n = m_1 + m_2 = \# X$.  By hypothesis, $\# N(X) < n = 
m_1 + m_2$; since $X \setminus Y \subset N(X) \setminus N(Y)$, we find 
$\# N(Y) < m_1 + m_2 - m_2 < m_1 = \# Y$. 
\end{proof}
\noindent
Combining Proposition \ref{ET2}, Theorem \ref{EXISTENCETHM} and 
Lemma \ref{ET3} we deduce the following result.
\begin{thm}(First Existence Theorem) 
\label{1ET}
\label{MAINEXISTENCETHM} \textbf{} \\
For a locally finite graph, TFAE: \\
(i) For every finite independent set $X$ in $G$, $\# X \leq \# N(X)$.  \\
(ii) $G$ admits a surjective graph derangement.
\end{thm}
\noindent
Remark 3.8: Theorem \ref{MAINEXISTENCETHM} was first proved by W.T. Tutte \cite[7.1]{Tutte53}.  Most of Tutte's paper is concerned with related -- but 
deeper -- results on digraphs.  The result which is our Theorem \ref{1ET} 
appears at the end of the paper and is proven by passage to an auxiliary 
digraph and reduction to previous results.  Perhaps because this was not 
the main focus of \cite{Tutte53}, Theorem \ref{1ET} seems not to be well known.  In particular, our observation that one need only apply Theorem \ref{HMT} appears to be new!

%
%
%
%

\subsection{Bipartite Existence Theorems}
\textbf{} \\ \\ \noindent
A \textbf{matching} on a graph $V = (G,E)$ is a subset 
$\mathcal{M} \subset E$ such that no two edges in $\mathcal{M}$ share a 
common vertex.  A matching $\mathcal{M}$ is \textbf{perfect} if every 
vertex of $G$ is incident to exactly one edge in $\mathcal{M}$.
\\ \\
A graph permutation is \textbf{dyadic} if 
all of its cycles have length at most $2$. 

\begin{prop}
\label{DYADICPROP}
\label{2.6}
Let $G$ be a graph.  \\
a) Matchings of $G$ correspond bijectively to dyadic graph permutations.  \\
b)  Under this bijection perfect matchings correspond to dyadic graph derangements.
\end{prop}
\begin{proof}
a) Let $\mathcal{M} \subset E$ be a matching.  We define $\sigma_{\mathcal{M}} \in \Sym V$ 
as follows: if $x$ incident to the edge $e = \{x, y\}$, we put $\sigma x = y$.  
Otherwise we put $\sigma x = x$.  This is well-defined since by definition 
every vertex is incident to at most one edge and gives rise to a dyadic 
graph permutation.  Conversely, to any dyadic graph permutation $\sigma \in 
\Sym V$, let $X \subset V$ be the subset of vertices which are not fixed by 
$\sigma$ and put $\mathcal{M}_{\sigma} = \bigcup_{x \in X} \{x, \sigma x\}$.  
Then $\mathcal{M} \mapsto \sigma_{\mathcal{M}}$ 
and $\mathcal{M} \mapsto \mathcal{M}_{\sigma}$ are mutually inverse.  \\
b) A matching $\mathcal{M}$ is perfect iff every vertex $x \in V$ is incident 
to an edge of $\mathcal{M}$ iff the permutation $\sigma_{\mathcal{M}}$ is 
fixed-point free.
\end{proof}
\noindent
Let $G = (V_1,V_2,E)$ be a bipartitioned graph.  A \textbf{semiperfect matching} on $G$ is a matching $\mathcal{M} \subset E$ such that every vertex of 
$V_1$ is incident to exactly one element of $\mathcal{M}$.  Thus a subset $\mathcal{M} \subset G$ is a perfect 
matching on $(V_1 \cup V_2,E)$ iff it is a semiperfect matching on both 
$(V_1,V_2,E)$ and on $(V_2,V_1,E)$.  \\ \indent
A \textbf{semiderangement} of $G = (V_1,V_2,E)$ is an injective function  $\iota: V_1 \ra V_2$ such that $x \sim \iota(x)$ for all $x \in V_1$.  \\ \indent
But in fact we have defined the same thing twice: a semiderangement of a 
bipartitioned graph is nothing else than a semiperfect matching.
 
\begin{thm}(Semiderangement Existence Theorem)
\label{BSDT}
\label{SEMIEXISTENCETHM}
Consider the following conditions on a bipartitioned graph $G = (V_1,V_2,E)$: \\
(SD) There is an injection $\iota: V_1 \ra V_2$ such that for all $x \in V_1$, 
$x \sim \iota(x)$.  \\
(H) For every finite subset $J \subset V_1$, $\# J \leq \# N(J)$. \\ 
Then (SD) $\implies$ (H), and if $G$ is locally finite, (H) $\implies$ (SD).
\end{thm}
\begin{proof} This is precisely Theorem \ref{HMTII} stated in the language of semiderangements.
\end{proof}

\begin{thm}(K\H onigs' Theorem)
\label{SYMTHM}
Let $G = (V_1,V_2,E)$ be a bipartitioned graph, which \emph{need not} be locally 
finite.  Suppose there is a semiderangement $\iota_1: V_1 \ra V_2$ of $(V_1,V_2,E)$ and a semiderangement $\iota_2: V_2 \ra V_1$ of $(V_2,V_1,E)$.  
Then $G$ admits a perfect matching, i.e., a bijection $f: V_1 \ra V_2$ such that $x \sim f(x)$ for all $x \in V_1$.
\end{thm}
\begin{proof} Let $V = V_1 \coprod V_2$.  Then $\iota = \iota_1 \coprod \iota_2: V \ra V$ is a graph derangement.  The injection $\iota$ partitions $V$ into  finite cycles, doubly infinite cycles, and singly infinite cycles.  For each 
cycle $C$ which is finite or doubly infinite, $\iota_1: C \cap V_1 \ra C \cap V_2$ and $\iota_2: C \cap V_2 \ra C \cap V_1$ are bijections.  
Thus if there are no singly infinite cycles, taking $f = \iota_1$ we are done.  If $C$ is a singly infinite cycle, it has an initial vertex $x_1$.  If $x_1 \in V_1$, then $\iota_1: C \cap V_1 \ra C \cap V_2$ 
is surjective; the problem occurs if $x_1 \in V_2$: then $x_1$ does not 
lie in the image of $\iota_1$.  We have $x_2 = \iota_2(x_1) \in V_1$, 
$x_3 = \iota_1(x_2) \in V_2$, and so forth.  So we can repair matters 
by defining $\iota_1$ on $x_2,x_4,\ldots$ by $x_2 \mapsto x_1$, $x_4 \mapsto x_3$, and so forth.  Doing this on every singly infinite cycle with initial 
vertex lying in $V_2$ a bijection $f: V_1 \ra V_2$.  Moreover, since $\iota_2$ is a semiderangement and $\iota(x_{2n-1}) = x_{2n}$ for all $n \in \Z^+$, 
we have $f(x_{2n}) = x_{2n-1} \sim x_{2n}$, so $f$ is a bijection.  
\end{proof}
\noindent
Remark 3.9: Suppose $V_1$ and $V_2$ are sets and $\iota_1: V_1 \ra V_2$ 
and $\iota_2: V_2 \ra V_1$ are injections between them.  If we apply Theorem 
\ref{SYMTHM} to the bipartitioned graph on $(V_1,V_2)$ in which $x \in V_1 \sim y \in V_2 \iff \iota_1(x) = y$ or $\iota_2(y) = x$, we get a 
bijection $f: V_1 \ra V_2$: this is the celebrated \textbf{Cantor-Bernstein Theorem}.  As a proof of Cantor-Bernstein, this argument was given by Gyula (``Julius'') K\H onig \cite{Konig06} and remains to this day one of the standard 
proofs.  His son D\'enes K\H onig explicitly made the connection to matching in infinite graphs in his seminal text \cite{Konig50}.

\begin{thm}(Second Existence Theorem)
\label{HALLISHTHM}
\label{2.8}
\label{2ET}
Let $G = (V_1,V_2,E)$ be a locally finite bipartitioned graph.  TFAE: \\
(i) $G$ admits a perfect matching. \\
(ii) $G$ admits a dyadic graph derangement.  \\
(iii) $G$ admits a graph derangement.  \\
(iv) For every subset $J \subset V$, $\# J \leq \# N(J)$.  
\end{thm}
\begin{proof}
(i) $\iff$ (ii) by Proposition \ref{DYADICPROP}.  \\
(ii) $\implies$ (iii) is immediate.  \\
(iii) $\implies$ (iv) is the same easy argument we have already seen.  \\ 
(iv) $\implies$ (i): By Theorem \ref{BSDT}, we have semiderangements 
$\iota_1: V_1 \ra V_2$ and $\iota_2: V_2 \ra V_1$.  By Theorem \ref{SYMTHM} this 
gives a perfect matching.   
%
\end{proof}
\noindent
Remark 3.10: The equivalence (i) $\iff$ (iv) above is due to R. Rado \cite{Rado49}.

\subsection{An Equivalence}
\textbf{} \\ \\ \noindent
Theorems \ref{1ET} and \ref{2ET} are ``equivalent'' in the sense that they were proved using equivalent formulations of Halls' Theorem (together with, in the case of Theorem \ref{2ET}, a Cantor-Bernstein argument).  In this section we 
will show their equivalence in a stronger sense: each can readily and rapidly 
be deduced from the other.
\\
\indent
Assume Theorem \ref{1ET}, and let $G = (V_1,V_2,E)$ be a locally finite bipartitioned graph satisfying the Cockroach Condition: for all finite independent subsets $J \subset V_1 \coprod V_2$, $\# N J \geq \# J$.  Then Theorem \ref{1ET} applies to yield a 
surjective graph derangement $f$.  A cycle $C$ admits a dyadic graph derangement iff it is not finite of odd length; since $G$ is bipartite, no cycle in $f$ 
is finite of odd length.  Thus decomposing $f$ cycle by cycle yields a dyadic 
graph derangement.  \\ 
\indent
Assume Theorem \ref{2ET}, and let $G = (V,E)$ be a locally finite graph satisfying the Hall Condition: for all finite independent subsets $J \subset V$, $\# NJ \geq \# J$.  Let $G_2 = (V_1,V_2,E_2)$ be the \textbf{bipartite double} of $G$: we put $V_1 = V_2 = V$.  For $x \in V$, let $x_1$ (resp. $x_2$) denote the copy of $x$ in $V_1$ (resp. $V_2$).  For every $e = \{x,y\} \in E$, we give ourselves 
edges $\{x_1,y_2\}, \{y_1,x_2\} \in E_2$.  Then $G_2$ is locally finite bipartitioned, and it is easy to see that the Hall Condition in $G$ implies 
the Cockroach Condition in $G_2$.  By Theorem \ref{2ET}, $G_2$ admits a dyadic 
graph derangement $f_2$.  From $f_2$ we construct a graph derangement $f$ of $G$: for $x \in V$, let $x_1$ be the corresponding element of $V_1$; let 
$y_2 = f_2(x_1)$, and let $y$ be the element of $V$ corresponding to $y_2$.  
Then we put $f(x) = y$.  It is immediate to see that $f$ is a graph derangement 
of $G$.  It need not be surjective, but no problem if it isn't: apply Proposition \ref{ET2}.
\\ \\
Remark 3.11: Our deduction of Theorem \ref{1ET} from Theorem \ref{2ET} is inspired by an unpublished manuscript of L. Levine \cite{Levine01}.
\\ \\
Remark 3.12: Recall that Theorem \ref{2ET} implies the 
Cantor-Bernstein Theorem.  The above equivalence thus has the following curious consequence: Cantor-Bernstein follows almost immediately from the 
transversal form of Halls' Theorem!

\subsection{Dyadic Existence Theorems}
\textbf{} \\ \\ \noindent
One may ask if there is also an Existence Theorem for dyadic derangements in 
non-bipartite graphs.  The answer is yes, a quite celebrated theorem of Tutte.  
A later result of Berge gives information on dyadic graph permutations.
\\ \\
Let $G = (V,E)$ be a graph.  For a subset $X \subset V$, we denote by 
$G \setminus X$ the induced subgraph with vertex set $V \setminus X$. 

\begin{thm}(Third Existence Theorem)
\label{2.9}
For a finite graph $G = (V,E)$, TFAE: \\
(i) $G$ has a dyadic graph derangement. \\
(ii) For every subset $X \subset V$, the number of connected components 
of $G \setminus X$ with an odd number of vertices is at most $\# X$.
\end{thm}
\begin{proof} See \cite{Tutte}.
\end{proof}
\noindent
Remark 3.13: Theorem \ref{2.9} can be generalized to locally finite graphs: 
see \cite{Tutte50}.
\\ \\
A \textbf{maximum matching} of a finite graph $G$ 
is a matching $\mathcal{M}$ such that $\# \mathcal{M}$ is maximized among 
all matchings of $G$.  Thus if $G$ admits a perfect matching, a matching $\mathcal{M}$ is a maximum matching iff it is perfect, whereas in general 
the size of a maximum matching measures the deviation from a perfect 
matching in $G$.  
\\ \indent
For a finite graph $H$, let $\operatorname{odd}(H)$ be the number of connected components of $H$ with an odd number of vertices.
\begin{thm}(Berge's Theorem \cite{Berge})
\label{2.10}
Let $G$ be any finite graph.  The size of a maximum matching in $G$ is 
\[ B_G = \frac{1}{2} \left( \min_{X \subset V} \# X - \operatorname{odd}(G \setminus X) + \# V \right). \]
\end{thm}
\noindent
We immediately deduce the following result on dyadic 
graph permutations.

\begin{cor}
Let $G$ be a finite graph.  Then the least number of fixed points in a 
dyadic graph permutation of $G$ is $\# V - 2 B_G$.
\end{cor}

\subsection{Matchless Graph Derangements}
\textbf{} \\ \\ \noindent
Let $G = (V,E)$ be a finite graph with $\# V = n$.  Then a graph permutation 
of cycle type $(n)$ is called a \textbf{Hamiltonian cycle} (or Hamiltonian 
circuit).  
\\ \\
From the perspective of graph derangements it is clear that Hamiltonian 
cycles lie at the other extreme from dyadic derangements and permutations.  Here much less is known than in the dyadic case: there is no known 
Hamiltonian analogue of Tutte's 
Theorem on perfect matchings, and in place of Berge's Theorem we have the 
following open question.
\begin{ques}
\label{2.11}
Given a finite graph $G$, determine the largest integer $n$ such that 
$G$ admits a graph permutation of type $(n,1,\ldots,1)$.  Equivalently, 
determine the maximum order of a vertex subgraph of $G$ admitting a 
Hamiltonian cycle.
\end{ques}
\noindent
A general graph derangement is close to being a partition of the vertex 
set into Hamiltonian cycles, \emph{except} that if $x$ and $y$ are adjacent vertices, 
then sending $x$ to $y$ and $y$ to $x$ meets the requirements of a graph 
derangement but does not give a Hamiltonian subcycle because the edge from 
$x$ to $y$ is being used twice.  Let us say a graph derangement is \textbf{matchless} if each cycle has length at least $3$.  The above Existence Theorems lead us naturally to the following question.
\begin{ques}
\label{MATCHLESSQUES}
Is there an Existence Theorem for matchless graph derangements?
\end{ques}
\noindent
As noted above, there is no known Existence Theorem for Hamiltonian cycles.  Since a Hamiltonian cycle is a graph derangement of a highly restricted kind, 
one might hope that Question \ref{MATCHLESSQUES} is somewhat more accessible.

\section{Checkerboards Revisited}

\subsection{Universal and even universal graphs}
\textbf{} \\ \\ \noindent
Let $G$ be a graph on the vertex set $[n] = \{1,\ldots,n\}$ such that 
$\Der G \neq \varnothing$.  As in $\S 2.2$ 
it is natural to inquire about the possible cycle types of graph derangements 
(and also graph permutations) of $G$.  We say $G$ is \textbf{universal} if for 
every partition $\mathfrak{p}$ of $n$ there is a graph derangement of $G$ with 
cycle type $n$.  For instance, the complete graph $K_n$ is (rather tautologously) universal.
\\ \\
If $n \geq 5$ and $G$ is bipartite, by Corollary \ref{2.5}, $G$ is \emph{not} universal, because the only possible cycle types are even.  Thus for 
graphs known to be bipartite the more interesting condition is that every 
possible even partition of $[n]$ occurs as the cycle type of a graph derangement 
of $G$: we call such graphs \textbf{even universal}.


\subsection{On the Even Universality of Checkerboard Graphs}

\begin{prop}
\label{4.1}
For all $n \in \Z^+$, the checkerboard graph $R_{2,n}$ is even 
universal.
\end{prop}
\begin{proof} This is an easy inductive argument which we leave to the reader.
\end{proof}

\begin{prop}
\label{4.2}
Let $n \geq 4$ be an even number.  Then: \\
a) If $a_1,\ldots,a_k$ are even numbers greater than $2$ such that 
$4 + \sum_{i=1}^k a_i = 3n$, then there is no graph derangement of 
$R_{3,n}$ of cycle type $(a_1,\ldots,a_k,4)$.  \\
b) It follows that $R_{3,n}$ is not even universal.
\end{prop}
\begin{proof}
a)  A $4$-cycle on any checkerboard graph must be a $(2 \times 2)$-square.  By 
symmetry we may assume that the $(2 \times 2)$-square is placed so as to occupy portions of the top two rows of $R_{m,n}$.  The two vertices immediately underneath the square cannot be part of any Hamiltonian cycle in the complement of the square, so any graph derangement containing a $4$-cycle must also contain a $2$-cycle directly underneath the $4$-cycle.  Thus the cycle type $(a_1,\ldots,a_k,4)$ is excluded.  \\
b) Since $n \geq 4$ is even, $3n$ is even and at least $12$, so
there are even partitions of $3n$ of the above form: $(4,4,\ldots,4)$ 
if $n$ is divisible by $4$; $(6,4,4,\ldots,4)$ otherwise.
\end{proof}

\begin{prop}
\label{4.3}
If $m$ is odd and $n$ is divisible by $4$, then $R_{m,n}$ admits 
no graph derangement of cycle type $(4,4,\ldots,4)$ and is therefore 
not even universal.
\end{prop}
\begin{proof} Left to the reader.
\end{proof}
\noindent
Example 4.1 ($G_{3,4}$): There are $p(\frac{12}{2}) = 11$ even partitions of 
$12$.  By Proposition \ref{4.3}, we cannot realize the cycle types $(8,4)$, 
$(4,4,4)$ by graph derangements.  We can realize the other nine:
\[ (12), (10,2), (8,2,2), (6,6), (6,4,2), (6,2,2,2), (4,4,2,2), (4,2,2,2,2), 
(2,2,2,2,2,2). \]
\\ 
Example 4.2 ($G_{4,4}$): Of the $p(\frac{16}{2}) = 22$ even partitions of $16$, 
we can realize $20$:
\[ (16), (14,2), (12,4), (12,2,2), (10,4,2), (10,2,2,2), (8,8), (8,6,2), (8,4,4),  (8,4,2,2), \] \[ (8,2,2,2), (6,6,2,2), (6,4,4,2), (6,4,2,2,2), (6,2,2,2,2,2), (4,4,4,4), (4,4,4,2,2), \] \[ (4,4,2,2,2,2), (4,2,2,2,2,2,2), (2,2,2,2,2,2,2,2,2). \]
We cannot realize: 
\[ (10,6), (6,6,4). \]
Indeed, the only order $6$ cycle of a checkerboard 
graph is the $2 \times 3$ checkerboard subgraph.  Removing any $6$-cycle leaves one of the corner vertices pendant and 
hence not part of any cycle of order greater than $2$.  
\\ \\
Example 4.3 ($G_{3,6}$): There are $p(\frac{18}{2}) = 30$ even partitions 
of $18$.  We can realize these $23$ of them: 
\[ (18), (16,2), (14,2,2), (12,6), (12,4,2), (12,2,2,2), (10,6,2), (10,4,2,2), 
\] \[
(10,2,2,2,2), (8,8,2), (8,6,2,2), (8,4,2,2,2), (8,2,2,2,2,2,2), (6,6,6), (6,6,4,2),  \] \[ (6,6,2,2,2), (6,4,4,2,2), (6,4,2,2,2,2), (6,2,2,2,2,2,2), (4,4,4,2,2,2), \] \[ (4,4,2,2,2,2,2), (4,2,2,2,2,2,2,2), 
(2,2,2,2,2,2,2,2,2). \]
We cannot realize the following ones:
\[ (14,4), (10,8), (10,4,4), (8,6,4), (8,4,4,2), (6,4,4,4), (4,4,4,4,2). \]
Of these, all but $(10,8)$ and $(4,4,4,4,2)$ are excluded by Proposition \ref{4.2}, and we leave it to the reader to check that these two ``exceptional'' cases cannot occur.  
\\ \\
Example 4.5 ($G_{4,5}$): There are $p(\frac{20}{2}) = 42$ even partitions 
of $20$.  We can realize $39$ of them.  We cannot realize: 
\[ (8,8,4), (8,4,4,4), (4,4,4,4,4)                       . \]
No $(8,8,4)$: to place a $4$-cycle in $G_{4,5}$ without leaving a pendant vertex, we must place it in a corner, without loss of generality the upper left corner.  The only $8$-cycle which can fit into the remaining lower left corner 
is the recantagular one, and this leaves a pendant vertex at the lower right corner.  \\
No $(8,4,4,4)$: Any placement of three $4$-cycles in $G_{4,5}$ gives two of them in either the top half or the bottom half, and without loss of generality the top half.  Any such placement leaves a pendant vertex in the top row.
\\
No $(4,4,4,4,4)$: This follows from Proposition \ref{4.3}.  Alternately, 
the argument of the previous paragraph works here as well.  
\\ \\
Example 4.6 ($G_{4,6}$): There are $p(\frac{24}{12}) = 77$ even partitions of 
$24$.  They can \emph{all} be realized by graph derangements: $G_{4,6}$ is 
even universal.  
\\ \\
By analyzing the above examples, we found the following additional families of excluded cycle types.  The proofs are left to the reader.

\begin{prop}
Let $n = 2k+1$ be an odd integer greater than $1$.  Then $G_{4,n}$ admits 
no matchless graph derangement with $2k-1$ or more $4$-cycles.
\end{prop}

\begin{prop}
\label{COOLJOSHPROP}
For any non-negative integer $k$, $G_{4,6k+4}$ admits no graph derangement 
of cycle type $(6,6,\ldots,6,4)$.  Thus these graphs are \emph{not} 
even universal.
\end{prop}

\begin{prop}
For any odd integer $m$ and $k \in \Z^+$, $G_{m,4k+2}$ admits no graph 
derangement of cycle type $(6,4,\ldots,4)$.  Thus these graphs are 
\emph{not} even universal.
\end{prop}
\noindent
In particular, for $R_{m,n}$ to be even universal, it is necessary 
that $m$ and $n$ both be even.  Proposition \ref{COOLJOSHPROP} shows 
that having $m$ and $n$ both even is not sufficient; however, we have 
found no examples of excluded even cycle types of $R_{m,n}$ when 
$m,n$ are both even and at least $6$.  Such considerations, and others, 
lead to the following conjecture, made by J. Laison in collaboration with 
the author.  

\begin{conj}
\label{LAISONCONJECTURE}
There is a positive integer $N_0$ such that for all $k,l \geq N_0$, the 
checkerboard graph $R_{2k,2l}$ is even universal.
\end{conj}



\end{document}